\documentclass{amsart}

\usepackage{amsmath,amssymb,amsthm, mathrsfs, mathtools, bm}
\usepackage{amscd,mathtools}
\usepackage[utf8]{inputenc}
\usepackage[english]{babel}
\usepackage{cite}
\usepackage{color}
\usepackage{hyperref}
\usepackage{float}
\usepackage{tikz}
\usetikzlibrary{automata}
\usepackage{comment}
\usepackage{mathtools}
\usepackage{csquotes}

\usepackage{amsthm}

\usepackage{tikz}
\usetikzlibrary{calc}
\usetikzlibrary{arrows}
\usetikzlibrary{decorations.pathreplacing}
\usetikzlibrary{intersections}

\usepackage{tikz}
\usepackage{caption, subcaption}
\usepackage{tikz-cd}
\usetikzlibrary{calc,decorations.pathreplacing}
\usepackage{tikz}
\usetikzlibrary{calc}
\usetikzlibrary{arrows}
\usetikzlibrary{decorations.pathreplacing}
\usetikzlibrary{intersections}

 \newtheorem{theorem}{Theorem}[section]
  \newtheorem{proposition}[theorem]{Proposition}
  \newtheorem{corollary}[theorem]{Corollary}
  \newtheorem{lemma}[theorem]{Lemma}

  \newtheorem{introthm}{Theorem}
  \newtheorem{introcor}[introthm]{Corollary}

  \theoremstyle{definition}
  \newtheorem{definition}[theorem]{Definition}

  \newtheorem*{claim*}{Claim}

  \newtheorem*{question*}{Question}
  \newtheorem*{answer*}{Answer}
  \newtheorem*{application*}{Application}

  \theoremstyle{remark}
  \newtheorem{remark}[theorem]{Remark}
  \newtheorem*{remark*}{Remark}

\usepackage{amssymb}

\DeclarePairedDelimiterX{\Norm}[1]{\lVert}{\rVert}{#1}
 \newcommand{\from}{\colon\thinspace} 
\theoremstyle{definition}

  \newcommand{\sC}{{\sf C}}   
  \newcommand{\sD}{{\sf D}}

  \newcommand{\sK}{{\sf K}}

  \newcommand{\sN}{{\sf N}}

  \newcommand{\sQ}{{\sf Q}}   
  \newcommand{\sR}{{\sf R}}

  \renewcommand{\aa}{{\sf a}}   
  \newcommand{\bb}{{\sf b}}   
  \newcommand{\ub}{{\underline{b}}}
    
  \newcommand{\cc}{{\sf c}}

    \newcommand{\E}{{\mathbb{e}}}

  \newcommand{\kk}{{\sf k}}   
       
  \newcommand{\mm}{{\sf m}}   
  \newcommand{\nn}{{\sf n}}

  \newcommand{\qq}{{\sf q}}   
  \newcommand{\rr}{{\sf r}}  
    \newcommand{\RR}{{\mathbb{R}}}

    \newcommand{\bfa}{{\textbf{a}}} 
  \newcommand{\bfb}{{\textbf{b}}}

  \newcommand{\gothic}{\mathfrak}
  \newcommand{\go}{{\gothic o}}

  \newcommand{\calC}{\mathcal{C}}

  \newcommand{\calN}{\mathcal{N}}

  \newcommand{\calU}{\mathcal{U}}

\newcommand{\CAT}{\ensuremath{\operatorname{CAT}(0)}\,}

\newcommand{\ou}{\pm}

  \newcommand{\ST}{\mathbin{\Big|}} 
\newcommand{\pka}{\partial_{\kappa}} 
\begin{document}

\title[Geometry and Dynamics on $\kappa$-Morse boundaries of \CAT groups]{Geometry and Dynamics on sublinearly Morse boundaries of \CAT groups}


 \author   {Yulan Qing}
 \address{Department of Mathematics, Fudan University,  Shanghai, China }
 \email{yulan.qing@gmail.com}
 
  \author   {Abdul Zalloum}
 \address{Department of Mathematics, Queen's University, Kingston, ON }
 \email{az32@queensu.ca}

 


\begin{abstract}
    Given a sublinear function $\kappa$, $\kappa$-Morse boundaries $\pka X$ of proper \CAT spaces are  introduced by Qing, Rafi and Tiozzo. It is a topological space that consists of a large set of quasi-geodesic rays and it is quasi-isometrically invariant and metrizable.
     In this paper, we study the sublinearly Morse boundaries with the assumption that there is a proper cocompact action of a group $G$ on the \CAT space in question. We show that $G$ acts minimally on $\pka G$ and that contracting elements of $G$ induces a weak north-south dynamic on $\pka G$. Furthermore, we show that a homeomorphism $f \from \pka G \to \pka G'$ comes from a quasi-isometry if and only if $f$ is successively quasi-m{\"o}bius and stable. Lastly, we characterize exactly when the sublinearly Morse boundary  of a \CAT space is compact.
\end{abstract}

\maketitle

\section{Introduction}
Much of the geometric group theory originates from the studying of hyperbolic groups and hyperbolic spaces. Hyperbolic groups have solvable word problem and strong dynamical properties. One fundamental technique in the study of hyperbolic groups is to study the Gromov boundaries of these groups. Gromov took the collection of all infinite geodesic rays (up to fellow traveling) in the associated Cayley graph, equipped this set with cone topology, and defined the space to be the boundary $\partial G$ of the hyperbolic group $G$. The boundary $\partial G$ is independent of the choice of a generating set and has rich geometric, topological, and algebraic structures (see for example the survey by Kapovich and Benakli \cite{Kapovich}).

If we view Gromov hyperbolic spaces as coarsely negatively curved, then the notion of \CAT  include spaces with both local and global non-positive curvature. Accordingly the extension of the boundary theory to \CAT spaces and groups has also been developing in recent decades. In this setting, the space of all geodesic rays together with the cone topology is called the visual boundary  (denoted by $\partial_{v} X$). It is shown by Croke and Kleiner  that the visual boundary of a \CAT space is not in general a quasi-isometry invariant \cite{Kleiner}. In \cite{Qin16}, it is shown that, in the Croke-Kleiner example, failure to obtain quasi-isometry invariance can come from  geodesic rays that spend linear amount of time (with respect to total time travelled) in each product region. 

Hence, one can consider geodesic rays that spend a sublinear amount of time in each product region. In \cite{QR19}, Qing and Rafi introduce the sublinearly Morse boundary $\pka X$ of a \CAT metric space $X$ and show that $\pka X$ is quasi-isometry invariant and metrizable. In \cite{QR19}, Qing and Tiozzo show that, for a right-angled Artin group $G$, $\pka G$ is a model for Poisson boundaries associated to a random walk $(G, \mu)$.  Intuitively, a (quasi-)geodesic ray is sublinearly Morse if it spends a sublinear amount of time in each maximal product region, with respect to total time travelled when it enters that product region. Furthermore, in \cite{GQR22}, it is shown that for every \CAT group with a rank one element, there exists a $\kappa$ such that $\pka G$ can be identified with the Poisson boundaries of the group. The authors of \cite{GQR22} also show that the sublinearly Morse directions in the visual boundary of a rank 1 \CAT space with a geometric group action are generic with respect to Patterson-Sullivan measures. Most recently, it is shown that much like the Gromov boundaries,  sublinearly Morse boundaries are sublinearly biLipschitz equivalence invariant \cite{QP}, providing a new way to tell when two groups are not sublinearly biLipschitz equivalence. These are evidences that the sublinearly Morse directions behalf similar to directions in hyperbolic spaces. In this paper, we continue to contribute to this comparison and focus on the dynamical property of the group action on $\kappa$-Morse boundaries. Much of the work in this paper is inspired by the methods in \cite{ChSu2014} and \cite{CashenMackay}. In more general proper geodesic spaces, sublinearly Morse boundaries have been developed and studied in \cite{QRT22, MQZ22, QN, IZ20, DZ22}, for instance.
  
\subsubsection*{Minimality of the group action}  A group is said to act \emph{minimally} on a topological space if every orbit is a dense subset of the space. We show that this property is enjoyed by the $\kappa$ boundaries. In contrast with the identifications with Poisson boundaries in various settings, the minimality result evidence the fact that the boundary is not too large in excess of the orbit of a point under the group action.


\begin{introthm}(Theorem~\ref{minimality theorem})\label{miminal-action-intro}
Suppose $G$ is a group that acts geometrically on a \CAT space $X$.  Then $G$ acts minimally on $\pka G$.
 \end{introthm}
 
 Based on this result, we illustrate that for a subset of the group elements, their actions induces the following form of north-south dynamics on the boundary:
 
\begin{introthm}(Theorem~\ref{NSdym})\label{NS-intro}
Suppose $G$ is a group that acts geometrically on a \CAT space $X$. Let $g \in G$ be a contracting element. For every open set $V$ containing $g^{\infty}$ and every compact set $C \in (\pka G \setminus [g^{\infty}])$, there exists an $N$ such that 
for all $n \geq N$, we have $g^nC \subset V$.
\end{introthm}

\subsubsection*{Compact type $\kappa$-boundaries.} In the examples shown in \cite{QR19}, the boundaries are not compact. We show that when $X$ is a proper hyperbolic space, $\kappa$-boundary is homeomorphic to the associated Gromov boundary. In fact, we show that this is exactly when a cocompact \CAT space $X$ has compact sublinear boundaries. On the other hand, examples of \CAT space without a cocompact group action, whose sublinearly boundaries are compact can be constructed easily. However it remains open to find a \CAT space $X$ with non-compact sublinear boundary where $\pka X$ is a perfect space.
When $X$ is a hyperbolic CAT(0) space, then the $\kappa$-boundary agrees with the Gromov boundary.
\begin{introthm} \label{compactness implies hyperbolicity intro}(Theorem~\ref{main theorem of the section})
Suppose a group $G$ acts geometrically on a proper \CAT space $X$ such that $\partial_{\kappa}X \neq \emptyset$, then the following are equivalent: 
\begin{enumerate}
    \item Every geodesic ray in $X$ is $\kappa$-contracting.
    \item Every geodesic ray in $X$ is strongly contracting.
    \item $\partial_{\kappa} X$ is compact.
    \item The space $X$ is hyperbolic.
\end{enumerate}

\end{introthm}

\begin{introcor}(Theorem \ref{main theorem of the section})
If $X$ is a proper CAT(0) hyperbolic space then $\pka X \simeq \partial X$. 
\end{introcor}

\subsubsection*{Rigidity} In 1996, Paulin gives the following characterization \cite{Paulin1996}: if $f : \partial X \to \partial Y$ is a homeomorphism between the boundaries of two proper, cocompact hyperbolic spaces, then the following are equivalent
\begin{enumerate}
\item $f$ is induced by a quasi-isometry $h: X \rightarrow Y$.
\item $f$ is quasi-m{\"o}bius.
\end{enumerate}

Quasi-m{\"o}bius maps are maps such that changes in the cross ratio are controlled by a continuous function. We aim to give a similar characterization for sublinearly Morse boundaries. We use the notion of \emph{successively quasi-m{\"o}bius} discussed it \cite{QR19}, which is a 1-parameter family of quasi-m{\"o}bius  maps on $\pka X$.

\begin{introthm}\label{introquasimobius} (Theorem \ref{quasimobius})
Let $X,Y$ be proper cocompact \CAT spaces with at least 3 points in their sublinear boundaries. A homeomorphism $f:\partial_\kappa X \to \partial_\kappa Y$ is induced by a quasi-isometry $h:X \to Y$ if and only if $f$ is stable and successively quasi-m{\"o}bius.
\end{introthm}

\subsection*{Acknowledgement} We thank Ruth Charney, Mathew Cordes, Curt Kent and Kasra Rafi for helpful conversations.

\section{Preliminaries}

\subsection{Quasi-isometry and quasi-isometric embeddings}

\begin{definition}[Quasi Isometric embedding] \label{Def:Quasi-Isometry} 
Let $(X , d_X)$ and $(Y , d_Y)$ be metric spaces. For constants $\kk \geq 1$ and
$\sK \geq 0$, we say a map $f \from X \to Y$ is a 
$(\kk, \sK)$--\textit{quasi-isometric embedding} if, for all points $x_1, x_2 \in X$
$$
\frac{1}{\kk} d_X (x_1, x_2) - \sK  \leq d_Y \big(f (x_1), f (x_2)\big) 
   \leq \kk \, d_X (x_1, x_2) + \sK.
$$
If, in addition, every point in $Y$ lies in the $\sK$--neighbourhood of the image of 
$f$, then $f$ is called a $(\kk, \sK)$--quasi-isometry. When such a map exists, $X$ 
and $Y$ are said to be \textit{quasi-isometric}. 

A quasi-isometric embedding $f^{-1} \from Y \to X$ is called a \emph{quasi-inverse} of 
$f$ if for every $x \in X$, $d_X(x, f^{-1}f(x))$ is uniformly bounded above. 
In fact, after replacing $\kk$ and $\sK$ with larger constants, we assume that 
$f^{-1}$ is also a $(\kk, \sK)$--quasi-isometric embedding, 
\[
\forall x \in X \quad d_X\big(x, f^{-1}f(x)\big) \leq \sK \qquad\text{and}\qquad
\forall y \in Y \quad d_Y\big(y, f\,f^{-1}(x)\big) \leq \sK.
\]
\end{definition}

A \emph{geodesic ray} in $X$ is an isometric embedding $\beta \from [0, \infty) \to X$. We fix a base-point $\go \in X$ and always assume 
that $\beta(0) = \go$, that is, a geodesic ray is always assumed to start from 
this fixed base-point. 
\begin{definition}[Quasi-geodesics] \label{Def:Quadi-Geodesic} 
In this paper, a \emph{quasi-geodesic ray} is a continuous quasi-isometric 
embedding $\beta \from [0, \infty) \to X$  starting from the basepoint $\go$. 
\end{definition}
The additional assumption that quasi-geodesics are continuous is not necessary for the results in this paper to hold,
but it is added for convenience and to make the exposition simpler. 

If $\beta \from [0,\infty) \to X$ is a $(\qq, \sQ)$--quasi-isometric embedding, 
and $f \from X \to Y$ is a $(\kk, \sK)$--quasi-isometry then the composition 
$f \circ \beta \from [t_{1}, t_{2}] \to Y$ is a quasi-isometric embedding, but it may 
not be continuous. However, one can adjust the map slightly to make it continuous 
(see Definition 2.2  \cite{QR19}) such that $f \circ \beta$ is a $(\kk\qq, 2(\kk\qq + \kk \sQ + \sK))$--quasi-geodesic ray.  

Similar to above, a \emph{geodesic segment} is an isometric embedding 
$\beta \from [t_{1}, t_{2}] \to X$ and a \emph{quasi-geodesic segment} is a continuous 
quasi-isometric embedding \[\beta \from [t_{1}, t_{2}] \to X.\] 

\noindent \textbf{Notation}. In this paper we will use $\alpha, \beta...$ to denote quasi-geodesic rays. If the quasi-geodesic constants are $(1, 0)$, we use $\alpha_{0}, \beta_{0},...$ to signify that they are in fact geodesic rays.  Meanwhile, we use $[\alpha], [\beta],...$ to denote equivalence classes of quasi-geodesic rays, and we also use $\bfa, \bfb,...$ to denote equivalence classes without referring an element in each class. Furthermore, let $\alpha$ be a (quasi-)geodesic ray $\alpha \from [0, \infty) \to X$, if $x_{1}, x_{2}$ are points on $\alpha$, then the segment of $\alpha$ between $x_{1}$ and $x_{2}$
is denoted $[x_{1}, x_{2}]_{\alpha}$. If a segment is presented without subscript, for example $[y_{1}, y_{2}]$, then it is a geodesic segment between the two points.
Let $\beta$ be a quasi-geodesic ray. For $\rr>0$, let $t_\rr$ be the first time where $\Norm{\beta(t)}=\rr$ and define:
\begin{equation}\label{notation}
\beta_\rr := \beta(t_\rr)
\qquad\text{and}\qquad
\beta|_{\rr} : = \beta{[0,t_\rr]} = [\beta(0),  \beta_{\rr}]_{\beta}
\end{equation}
which are points and segments in $X$, respectively. 
\subsection{Properties of \CAT spaces}
A geodesic metric space $(X, d_X)$ is \CAT if geodesic triangles in $X$ are at 
least as thin as triangles in Euclidean space with the same triple of side-lengths. To be precise, for any 
given geodesic triangle $\triangle pqr$, consider the unique triangle 
$\triangle \overline p \overline q \overline r$ in the Euclidean plane with the same triple of side-lengths. The triangle $\triangle pqr$ is at least at thin as $\triangle \overline p \overline q \overline r$  in the following
sense: For any pair of points $x, y$ on the triangle $\triangle pqr$, without loss of generality let $x, y$ be on edges $[p,q]$ and $[p, r]$, if we choose points $\overline x$ and $\overline y$  on 
edges $[\overline p, \overline q]$ and $[\overline p, \overline r]$ of 
the triangle $\triangle \overline p \overline q \overline r$ so that 
$d_X(p,x) = d_\E(\overline p, \overline x)$ and 
$d_X(p,y) = d_\E(\overline p, \overline y)$ then,
\[ 
d_{X} (x, y) \leq d_{\mathbb{E}^{2}}(\overline x, \overline y).
\] 


%
%
%
%
%
%
%
%
%
%
%
%
%
%
%
%
%
%
%
%
%
%
%
%
%
%
%
%
%
%
%
%
%
%
%
%
%
%
%
%
%
%
%
%
%
%
%
%
%
%
%
%
%
%
%
%
%
%
%
%
%
%
%
%
%
%
%
%
%
%
%
%
%
%
%
%
%
%
%
%
%
%
A metric space $X$ is {\it proper} if closed metric balls are compact. For the remainder of the paper, we assume $X$ is a proper \CAT space;  a proper \CAT space has the following basic properties that are needed in this paper:
\begin{lemma}
 A proper \CAT space $X$ has the following properties:
\begin{enumerate}
\item For any two points $x, y$ in $X$, 
there exists exactly one geodesic connecting them. Consequently, $X$ is contractible 
via geodesic retraction to a base point in the space. 
\item The nearest point projection from a point $x$ to a geodesic line $\beta_{0}$ 
is a unique point denoted $\pi_{\beta_{0}}(x)$, or simply $x_{\beta_{0}}$. In fact, the closest point projection map to a geodesic 
\[
\pi_{\beta_{0}} \from X \to \beta_{0}
\]
is Lipschitz with respect to distances. The nearest point projection from a point $x$ to a quasi-geodesic line $\beta$ exists and is not necessarily unique. We denote the whole projection set $\pi_{\beta}(x)$.
\item For any $x \in X$, the distance function $d_{X}(x,\cdot)$ is convex. In other words, for any given any geodesic $[x_0,x_1]$ and  $t \in [0,1]$, if $x_t $  satisfies $d_{X}(x_0,x_t)=td(x_0,x_1)$ then we must have 
\[d_{X}(x,x_t) \leq (1-t)d_{X}(x,x_0)+td_{X}(x,x_1).\]
\end{enumerate}
\end{lemma}

In addition, we need the following redirecting surgery for all proper metric spaces:
%

\begin{lemma}\label{Lem:surgery}
Let  $X$ be a proper, complete metric space. Let $b$ be a geodesic ray and $\gamma$ be a $(\qq, \sQ)$--quasi-geodesic ray. 
For $\rr>0$, assume that $d_X(b_\rr, \gamma)\leq \rr/2$. Then, there exists a
$(9\qq,\sQ)$--quasi-geodesic $\gamma'$ so that 
\[
\gamma' \in [b], \qquad\text{and}\qquad \gamma|_{\rr/2} = \gamma'|_{\rr/2}. 
\]
\end{lemma}

%

\subsection{Boundaries of \CAT space}
\subsubsection{Visual boundary}

In this section we review a couple of topological boundaries of \CAT spaces that are important to the study of this paper.

 \begin{definition}[visual boundary]
Let $X$ be a \CAT space. The \emph{visual boundary} of $X$, denoted $\partial_{v} X$,  is the collection of equivalence classes of infinite geodesic rays, where $\alpha$ and $\beta$ are in the same equivalence class, if and only if there exists some $C \geq 0$ such that $d(\alpha(t), \beta(t)) \leq C$ for all $t \in [0 ,\infty).$ The equivalence class of $\alpha$ in $\partial_v X$ we denote $\alpha(\infty)$.
 \end{definition}
  Notice that by Proposition I. 8.2 in \cite{BH1}, for each $\alpha$ representing an element of $\partial X$, and for each $x' \in X$, there is a unique geodesic ray $\alpha'$ starting at $x'$ with $\alpha(\infty)=\alpha'(\infty).$

We describe the topology of the visual boundary by a neighbourhood basis: fix a base point $\go$ and let $\alpha$ be a geodesic ray starting at $\go$. A neighborhood basis for $\alpha$ is given by sets of the form: 
\[\calU_{v}\big(\alpha(\infty), r, \epsilon):=\{ \beta(\infty) \in \partial_{v} X | \, \beta(0)=\go\,\,\text{and }\,d(\alpha(t), \beta(t))<\epsilon \, \, \text{for all }\, t < r \}.\]

 In other words, two geodesic rays are close if they have
geodesic representatives that start at the same point and stay close (are at most $\epsilon$ apart) for a long
time (at least $r$). Notice that the above definition of the topology on $\partial_{v} X$ references a base-point $\go$. Nonetheless, Proposition I. 8.8 in \cite{BH1} proves that the topology of the visual boundary is base-point invariant. 

\begin{definition}(Visibility) A point $\zeta$ in the visual boundary is said to be a \emph{visibility point} if any other point $\zeta' \in \partial \partial X,$ there exists a geodesic line $l$ with $l(\infty)=\zeta$ and $l(-\infty)=\zeta'.$ A subset $Y \subseteq \partial X$ is said to be a \emph{visibility space} if for any $\zeta,\zeta' \in Y$ with $\zeta \neq \zeta',$ there is a geodesic line $l$ with $l(\infty)=\zeta$ and $l(-\infty)=\zeta'.$
 
\end{definition}

Related to the above, in \cite{Zalloum_2022}, it's shown that each point of the sublinearly Morse boundary $\partial_\kappa X$ is a visibility point of $\partial X$.

\begin{figure}[h]
\begin{center}
\begin{tikzpicture}[scale=0.6]

\node (x) [circle,fill,inner sep=1pt,label=180:$\go$] at (0,0) {};
\draw [name path=circle] (0,0) circle (3);

\draw [name path=line,thin] (0,0) to [bend right=20] (5,2.5);
\draw [thin] (0,0) to [bend left=20] (5,-2.5);
\draw (0,0) to (5,0) node [right] {$\alpha$};

\draw [very thin,
	decorate,
	decoration={brace,mirror,amplitude=12pt}] 
	(0,0) -- (3,0) node [midway,yshift=-20pt] {$r$};

\draw [very thin,
	name intersections={of=line and circle},
	decorate,
	decoration={brace,amplitude=4pt}] 
	(intersection-1) -- (3,0) node [midway,xshift=12pt] {$\epsilon$};

\end{tikzpicture}
\end{center}
\caption{A basis for open sets}
\label{}
\end{figure}

\subsubsection{Sublinearly Morse boundaries} \hfill

Let $\kappa \from [0, \infty) \to [1, \infty)$ be a sublinear function that is monotone increasing and concave. That is
\[
\lim_{t \to \infty} \frac{\kappa(t)}{t} = 0. \label{subfunction}
\]

The assumption that $\kappa$ is increasing and concave makes certain arguments
cleaner, otherwise they are not really needed. One can always replace any 
sub-linear function $\kappa$, with another sub-linear function $\overline \kappa$
so that \[\kappa(t) \leq \overline \kappa(t) \leq \sC \, \kappa(t)\] for some constant $\sC$ 
and $\overline \kappa$ is monotone increasing and concave. For example, define 
\[
\overline \kappa(t) = \sup \Big\{ \lambda \kappa(u) + (1-\lambda) \kappa(v) \ST 
\ 0 \leq \lambda \leq 1, \ u,v>0, \ \text{and}\ \lambda u + (1-\lambda)v =t \Big\}.
\]
The requirement $\kappa(t) \geq 1$ is there to remove additive errors in the definition
of $\kappa$--contracting geodesics(See Definition~\ref{Def:Contracting}). 

%
\subsubsection{$\kappa$--Morse geodesic rays}
The boundary of interest in this paper consists of points in $\partial X$ that are in the ``hyperbolic-like''. In proper \CAT spaces, they can be characterized in two equivalence ways.

\begin{definition}[$\kappa$--neighborhood]  \label{Def:Neighborhood} 
For a closed set $Z$ and a constant $\nn$ define the $(\kappa, \nn)$--neighbourhood 
of $Z$ to be 
\[
\calN_\kappa(Z, \nn) = \Big\{ x \in X \ST 
  d_X(x, Z) \leq  \nn \cdot \kappa(x)  \Big\}.
\]

\begin{figure}[h]
\begin{tikzpicture}
 \tikzstyle{vertex} =[circle,draw,fill=black,thick, inner sep=0pt,minimum size=.5 mm] 
[thick, 
    scale=1,
    vertex/.style={circle,draw,fill=black,thick,
                   inner sep=0pt,minimum size= .5 mm},
                  
      trans/.style={thick,->, shorten >=6pt,shorten <=6pt,>=stealth},
   ]

 \node[vertex] (a) at (0,0) {};
 \node at (-0.2,0) {$\go$};
 \node (b) at (10, 0) {};
 \node at (10.6, 0) {$b$};
 \node (c) at (6.7, 2) {};
 \node[vertex] (d) at (6.68,2) {};
 \node at (6.7, 2.4){$x$};
 \node[vertex] (e) at (6.68,0) {};
 \node at (6.7, -0.5){$x_{b}$};
 \draw [-,dashed](d)--(e);
 \draw [-,dashed](a)--(d);
 \draw [decorate,decoration={brace,amplitude=10pt},xshift=0pt,yshift=0pt]
  (6.7,2) -- (6.7,0)  node [black,midway,xshift=0pt,yshift=0pt] {};

 \node at (7.8, 1.2){$\nn \cdot \kappa(x)$};
 \node at (3.6, 0.7){$||x||$};
 \draw [thick, ->](a)--(b);
 \path[thick, densely dotted](0,0.5) edge [bend left=12] (c);
\node at (1.4, 1.9){$(\kappa, \nn)$--neighbourhood of $b$};
\end{tikzpicture}
\caption{A $\kappa$-neighbourhood of a geodesic ray $b$ with multiplicative constant $\nn$.}
\end{figure}
\end{definition}

\begin{definition} [$\kappa$-Morse I, $\kappa$-Morse II] \label{D:k-morse} \label{Def:Morse} 
Let $Z \subseteq X$ be a closed set, and let $\kappa$ be a concave sublinear function. 
We say that $Z$ is \emph{$\kappa$-Morse} if one of the following equivalent (see Proposition 3.10 \cite{QRT22})  condition holds:
\begin{enumerate}

\item[I] There exists a proper function 
$\mm_Z : \mathbb{R}^2 \to \mathbb{R}$ such that for any sublinear function $\kappa'$ 
and for any $r > 0$, there exists $R$ such that for any $(\qq, \sQ)$-quasi-geodesic ray $\beta$
with $\mm_Z(\qq, \sQ)$ small compared to $r$, if 
$$d_X(\beta_R, Z) \leq \kappa'(R)
\qquad\text{then}\qquad
\beta|_r \subset \calN_\kappa \big(Z, \mm_Z(\qq, \sQ)\big)$$

\item[II] There is a function
\[
\mm'_Z \from \RR_+^2 \to \RR_+
\]
so that if $\beta \from [s,t] \to X$ is a $(\qq, \sQ)$--quasi-geodesic with end points 
on $Z$ then
\[
[s,t]_{\beta}  \subset \calN_{\kappa} \big(Z,  \mm'_Z(\qq, \sQ)\big). 
\]
\end{enumerate}
\end{definition} 

\begin{remark}
By taking the maximum function of $\mm_Z, \mm'_Z,$ we may and will always assume that both conditions hold for the same $\mm_Z.$ Further,
\begin{equation} 
\mm_Z(\qq, \sQ) \geq \max(\qq, \sQ). 
\end{equation} 
\end{remark}

\begin{definition}[$\kappa$--contracting sets] \label{Def:Contracting}
For $x \in X$, define $\Norm{x} = d_X(\go, x)$. 
For a closed subspace $Z$ of $X$, we say $Z$ is \emph{$\kappa$--contracting} if there 
is a constant $\cc_Z$ so that, for every $x,y \in X$
\[
d_X(x, y) \leq d_X( x, Z) \quad \Longrightarrow \quad
diam_X \big( x_Z \cup y_Z \big) \leq \cc_Z \cdot \kappa(\Norm x).
\]
In fact, to simplify notation, we drop $\Norm{\cdot}$ when it appears in the $\kappa$ function and write  $\kappa(x)$ instead of  $\kappa(\Norm{x})$. 

\end{definition} 

\begin{figure}[H]
\begin{tikzpicture}[scale=0.8]
\draw[thick] (0,0) -- (8,0);

\node[left] at (0,0) {$\go$};

\draw[red, dashed] (5.32,2.3)  .. controls ++(-.3,-.4) and ++(0,1) .. (4.5,0) ;

\draw[red, dashed] (2.68,2.3)  .. controls ++(.3,-.4) and ++(0,1) .. (3.5,0) ;

\draw[very thick, red] (3.5,0) -- (4.5,0);

\node[below] at (4,0) {$ \leq \cc_{\alpha} \kappa(||x||)$};

\node[below] at (8,0) {$\alpha$};

\draw[thick,red] (4,3) circle (1.5cm);

\node[ left] at (4,3) {$x$};

\draw[thin,dashed, ->] (0,0) -- (4,3);

\node[ left] at (2,2) {$||x||$};

\draw[thick,fill=black] (4,3) circle (0.05cm);

\end{tikzpicture}
\caption{A $\kappa$--contracting geodesic ray.}
\end{figure}

In \CAT     spaces, a geodesic is $\kappa$--contracting if and only if it is $\kappa$--Morse\cite{QR19}. 

$\kappa$-Morse quasi-geodesic rays in $X$ are grouped into equivalence classes to form $\pka X$.
\begin{definition}[$\kappa$--equivalence classes in $\pka X$] \label{Def:Fellow-Travel}
Let $\beta$ and $\gamma$ be two quasi-geodesic rays in $X$. If $\beta$ is in some 
$\kappa$--neighbourhood of $\gamma$ and $\gamma$ is in some 
$\kappa$--neighbourhood of $\beta$, we say that $\beta$ and $\gamma$ 
\emph{$\kappa$--fellow travel} each other. This defines an equivalence
relation on the set of quasi-geodesic rays in $X$ (to obtain transitivity, one needs to change $\nn$ of the associated $(\kappa, \nn)$--neighbourhood). 
\end{definition}

We denote the equivalence class 
that contains $\beta$ by $[\beta]$:
\begin{definition}[Sublinearly Morse boundary]
Let $\kappa$ be a sublinear function as specified in Section~\ref{subfunction} and let $X$ be a \CAT space.
\[\pka X : = \{ \text{ all } \kappa\text{-Morse quasi-geodesics } \} / \kappa\text{-fellow travelling}\]
We define the topology of $\pka X$ in Section~\ref{coarsetop}.
\end{definition}

We also use $\bfa, \bfb$ to denote $\kappa$-equivalence classes in $\pka X$. We need the following fact that since $X$ is \CAT, there is a unique geodesic ray in each equivalence class:

\begin{lemma}[Lemma 3.5, \cite{QR19}]\label{uniquegeodesic} 
Let $X$ be a \CAT space.  Let $b \from [0,\infty) \to X$ be a geodesic ray in $X$. Then $b$ is the unique geodesic 
ray in any $(\kappa, \nn)$--neighbourhood of $b$ for any $\nn$. That is to say,  there is an 1-1 embedding of the set of points in $\pka X$ into the points of $\partial_{v} X$.
\end{lemma}

\begin{proof}
For each element $\bfa \in \pka X$, consider its unique geodesic ray $\alpha$. The associated $\alpha(\infty)$ is a element of $\partial_{v} X$. By Lemma 3.5, \cite{QR19}, each equivalence class contains a unique geodesic ray.  Meanwhile, if two elements $\bfa, \bfb \in \pka X$ contain the same geodesic ray, they are in fact the same set of quasi-geodesics, therefore this map is well-defined.  Therefore we have an embedding of the set of points in $\pka X$ into the points of $\partial_{v} X$.
\end{proof}

\subsubsection{Coarse cone topology on $\pka X$}\label{coarsetop}

 We equip $\partial_\kappa X$ with a topology which is 
a coarse version of the visual topology. In visual topology, if two geodesic rays fellow travel for a long time, then they are ``close''. In this coarse version, if two geodesic rays and all the quasi-geodesic rays in their respect equivalence classes remain close for a long time, then they are close. Now we define it formally. First, we say a quantity $\sD$ \emph{is small compared to a radius $\rr>0$} if 
\begin{equation} \label{Eq:Small} 
\sD \leq \frac{\rr}{2\kappa(\rr)}. 
\end{equation} 

%
Recall that given a $\kappa$--Morse quasi-geodesic ray $\beta$, we denotes its associated Morse gauges functions $\mm_{\beta}(\qq, \sQ)$. These are multiplicative constants that give the heights of the $\kappa$--neighbourhoods.
\begin{definition}[topology on $\pka X$]
Let $\bfa \in \partial_\kappa X$ and $\alpha_0 \in \bfa$ be the unique geodesic in the class $\aa.$ Define $ \calU_{\kappa}(\bfa, \rr)$ to be the set of points $\bfb$ such that for any $(\qq, \sQ)$-quasi-geodesic of $\bfb$, denoted  $\beta$, 
such that $\mm_{\beta}(\qq, \sQ)$ is small compared to $\rr$, satisfies 
\[
\beta|_{\rr} \subset \calN_{\kappa}\big(\alpha_{0}, \mm_{\alpha_{0}}(\qq, \sQ)\big).\]

Let the topology of $\pka X$ be the topology induced by this neighbourhood system. The following fact shows that a $\kappa$-boundary is well defined with respect the associated group.
\end{definition}

\begin{figure}[H]
\begin{tikzpicture}[scale=0.5]
 \tikzstyle{vertex} =[circle,draw,fill=black,thick, inner sep=0pt,minimum size=.5 mm]
 
[thick, 
    scale=1,
    vertex/.style={circle,draw,fill=black,thick,
                   inner sep=0pt,minimum size= .5 mm},
                  
      trans/.style={thick,->, shorten >=6pt,shorten <=6pt,>=stealth},
   ]

  \node[vertex] (o) at (-10,0)  [label=left:$\go$] {}; 
  \node (o1) at (2, 5)[label=right:\color{red} geodesic $\beta_{0}$] {}; 
  \draw[thick, red]  (o) to [ bend right=20] (o1){};
  \node[vertex] (a) at (7,0)  [label=below:$\alpha_{0}$] {};
 \path[thick, densely dotted](-10,0.7) edge [bend left=12] (4, 3){};
 \node at (5.8, 2.5){$(\kappa, \mm_{\alpha_{0}}(\qq, \sQ))$-neighbourhood of $\alpha_{0}$};

       \node (a1) at (-1, 5) {}; 
   \node at (6,4) { $(\qq, \sQ)$--quasi-geodesic $\beta$};  
  \draw[dashed]  (-0.5, 5) to  [bend right=5] (-0.5, -1){};
  \node at (-0.5, -1){$\rr$};

   \draw[thick]  (o)--(a){};

  \pgfsetlinewidth{1pt}
  \pgfsetplottension{.75}
  \pgfplothandlercurveto
  \pgfplotstreamstart
  \pgfplotstreampoint{\pgfpoint{-10cm}{0cm}}  
  \pgfplotstreampoint{\pgfpoint{-9cm}{1cm}}   
  \pgfplotstreampoint{\pgfpoint{-8cm}{0.5cm}}
  \pgfplotstreampoint{\pgfpoint{-7cm}{1.5cm}}
  \pgfplotstreampoint{\pgfpoint{-6cm}{1cm}}
  \pgfplotstreampoint{\pgfpoint{-5cm}{2cm}}
  \pgfplotstreampoint{\pgfpoint{-4cm}{1.1cm}}
  \pgfplotstreampoint{\pgfpoint{-3cm}{1cm}}
  \pgfplotstreampoint{\pgfpoint{-2cm}{1.5cm}}
  \pgfplotstreampoint{\pgfpoint{-1cm}{1cm}}
    \pgfplotstreampoint{\pgfpoint{-0.5cm}{2cm}}
      \pgfplotstreampoint{\pgfpoint{-0.1cm}{3cm}}
        \pgfplotstreampoint{\pgfpoint{1cm}{4cm}}
          \pgfplotstreampoint{\pgfpoint{2cm}{4cm}}
  \pgfplotstreamend 
  \pgfusepath{stroke}
       
  \end{tikzpicture}
 
  \caption{ $\bfb \in \calU_{\kappa}(\bfa, \rr)$ because the quasi-geodesics of $\bfb$ such as $\beta, \beta_{0}$ stay inside the associated $(\kappa, \mm_{\alpha_{0}}(\qq, \sQ))$-neighborhood of $\alpha_{0}$ (as in Definition~\ref{Def:Neighborhood} ), up to distance $\rr$. }
 \end{figure}

\begin{theorem} [\cite{QRT22}]
Let $X$ be a proper metric space and let $\kappa$ be a sublinear function. The $\kappa$--boundary of $X$, denoted $\pka X$, is a quasi-isometrically invariant topological space. Furthermore, $\pka X$ is metrizable.
\end{theorem}

\begin{proposition} [Proposition 4.4 \cite{QRT22}] \label{Prop:Normal}
For each $\bfb \in \partial_\kappa X$ and $\rr > 0$, there exists a radius 
$\rr_\bfb$ such that 
\begin{enumerate} 
\item for any point $\bfa$ there exists $\rr_\bfa$ so that 
\[
\bfa \in \calU_{\kappa}(\bfb, \rr_\bfb) 
\qquad\Longrightarrow\qquad
\calU_{\kappa}(\bfa, \rr_\bfa) \subset \calU_{\kappa}(\bfb, \rr).
\]
\item for any point $\bfa$ there exists $\rr_\bfa$ so that
\[
\bfa  \notin \calU_{\kappa}(\bfb, \rr) 
\qquad\Longrightarrow\qquad
\calU_{\kappa}(\bfa, \rr_\bfa) \cap \calU_{\kappa}(\bfb , \rr_\bfb) = \emptyset. 
\]
\end{enumerate} 
\end{proposition}

\section{Dense subsets and minimality of $G$-action}

In this section we prove two results concerning dense subsets of $\pka G$. First we show that the set of all Morse directions, $\partial_{1} G$ is dense in $\pka G$, secondly and more generally, the action of $G$ is minimal on $\pka G$ and as a consequence,   a Morse element in  $G$ acts with North-South dynamics on the boundary. To begin with,  in this section, let $G$ acts geometrically on a \CAT space $X$.

Let $A$ be a geodesic ray or a geodesic segment. We say that a sequence of group elements $\{g_i \}$ \emph{tracks} $A$ if there exists a strict fundamental domain $D$ such that the union $\cup_i g_i\cdot D$ covers longer and longer subsegment of $A$ as $i$ increases. Conventionally we order the elements such that $\cup_{i \leq n} g_i\cdot D$ covers a longer and longer segment of $A$ as $n$ increases.

\begin{definition}[Angles in \CAT spaces]]\cite[II.3.1]{BH1}
 Let $X$ be a $\CAT$ space and let $\ell \from [0,a] \to X$ and $\ell' \from [0, a'] \to X$ be two geodesic paths issuing from the same point $\ell(0) = c'(0)$.
Then the comparison angle $\angle_{\mathbb{E}}(c(t), c'(t'))$ is a non-decreasing function of both $
t,t' \geq 0$, and the \emph{Alexandrov angle} $\angle(c,c')$ is equal to 
\[ \lim_{t,t'\to 0} \angle_{c(0)} (c(t),c'(t')) = \lim_{t \to 0} \angle (c(t), c'(t)).\] Hence, we define:
\[ \angle (c, c') = \lim_{t \to 0} 2 \arcsin \frac{1}{2t} d(c(t), c'(t)).\]

We also refer to the Alexandrov angle as the local angle.
\end{definition}
\begin{lemma}\label{concatenationofgauges}
Let $c_0$ be a concatenation of  a geodesic segment and an infinite geodesic ray as follows: 
\begin{itemize}
    \item the geodesic segment is the initial segment of a $\kappa$-Morse geodesic ray  labelled $\underline{b}$;
    \item the infinite geodesic ray is a $\kappa$-Morse geodesic ray labelled $\underline{a}$.
    \item Suppose in addition that the Alexandrov angle are the point of concatenation is bounded below by the right angle.
    \begin{figure}[H]
\begin{tikzpicture}[scale=0.5]
 \tikzstyle{vertex} =[circle,draw,fill=black,thick, inner sep=0pt,minimum size=.5 mm]
 
[thick, 
    scale=1,
    vertex/.style={circle,draw,fill=black,thick,
                   inner sep=0pt,minimum size= .5 mm},
                  
      trans/.style={thick,->, shorten >=6pt,shorten <=6pt,>=stealth},
   ]

  \node[vertex] (o) at (0,0)  [label=left:$\go$] {}; 
  \node [vertex] (o1) at (5, 0){}; 
    \node(o2) at (9, 4){}; 
  \node at (8, 3)[label=right: $\underline{a}$] {};
   \node at (2.3, 0)[label=below right: $\underline{b}$] {};
   \node at (4.7, 0.9) {\small $\geq \frac{\pi}{2}$};
  \draw[thick, black]  (4.5, 0) to [ bend left=20] (5.3, 0.3){};
  \draw[thick, black] (o)--(o1)--(o2){};
%
%
%

  \end{tikzpicture}
 
  \caption{ The local angle in a \CAT space.}
 \end{figure}

\end{itemize}
Then $c_0$ is $\kappa$-Morse, and its Morse gauge is bounded above by
\[
m_{c_0}(q, Q) \leq m'_{\underline{b}}(3q, Q) + m_{\underline{a}}(3q, Q)
\]

\end{lemma}
\begin{proof} We will show that $c_0$ satisfies the $\kappa$-Morse II condition.
Consider a quasi-geodesic ray $c'$ that sublinearly tracks $c_0$.
Let $p$ be the point of concatenation on $c_0$ and project $p$ to $c'$ and label the projection as $p_{c'}$. By The Surgery Lemma \cite[Lemma 2.5]{QR19},
\[ [\go, p_{c'}]_{c'} \cup [p, p_{c'}]
\]

is a $(3q, Q)$ quasi-geodesic segment whose endpoints are on $\underline{b}$, thus by $\kappa$-Morse II it is in the $m'_{\underline{b}}(3q, Q)$ neighbourhood of $\underline{b}$. Define $c''$ to to be the quasi-geodesic ray with $c'' \subseteq c'$ and $c''(0)=p_{c'}$. Again, by The Surgery Lemma \cite[Lemma 2.5 ]{QR19},
\[ [p, p_{c'}] \cup c'' \]
is $(3q, Q)$ quasi-geodesic ray that sublinearly tracks $\underline{a}$. Likewise by $\kappa$-Morse I it is in the  $m_{\underline{a}}(3q, Q)$ neighbourhood of $\underline{a}$. Thus $c'$ is in a $m'_{\underline{b}}(3q, Q) + m_{\underline{a}}(3q, Q)$-neighbourhood of $c_0$.
\end{proof}

   \begin{figure}[H]
\begin{tikzpicture}[scale=0.5]
 \tikzstyle{vertex} =[circle,draw,fill=black,thick, inner sep=0pt,minimum size=.5 mm]
 
[thick, 
    scale=1,
    vertex/.style={circle,draw,fill=black,thick,
                   inner sep=0pt,minimum size= .5 mm},
                  
      trans/.style={thick,->, shorten >=6pt,shorten <=6pt,>=stealth},
   ]

  \node[vertex] (o) at (0,0)  [label=left:$\go$] {}; 
  \node [vertex] (o1) at (5, 0)[label=below:$p$] {}; 
    \node(o2) at (9, 4){}; 
  \node at (8, 3)[label=right: $\underline{a}$] {};
    \node at (7.5, 5.5)[label=right: $c''$] {};
     \node at (3.3, 4)[label=right: $c'$] {};
   \node at (2.3, 0)[label=below right: $\underline{b}$] {};
   \node  [vertex] at (4, 2)[label=below left: $p_{c'}$] {};
    \draw[thick, dashed] (o1)--(4,2){};

  \draw[thick, black] (o)--(o1)--(o2){};

         \pgfsetlinewidth{1pt}
  \pgfsetplottension{.75}
  \pgfplothandlercurveto
  \pgfplotstreamstart
  \pgfplotstreampoint{\pgfpoint{0cm}{0cm}}  
  \pgfplotstreampoint{\pgfpoint{1cm}{2cm}}   
  \pgfplotstreampoint{\pgfpoint{2cm}{1.5cm}}
  \pgfplotstreampoint{\pgfpoint{3cm}{3cm}}
  \pgfplotstreampoint{\pgfpoint{4cm}{2cm}}
  \pgfplotstreampoint{\pgfpoint{5cm}{3cm}}
  \pgfplotstreampoint{\pgfpoint{5.5cm}{2.8cm}}
    \pgfplotstreampoint{\pgfpoint{6.5cm}{4.5cm}}
        \pgfplotstreampoint{\pgfpoint{7.5cm}{4cm}}
          \pgfplotstreampoint{\pgfpoint{7.5cm}{5.5cm}}

  \pgfplotstreamend 
  \pgfusepath{stroke}

  \end{tikzpicture}
  \caption{The quasi-geodesic ray $c'$ is covered by the union of a quasi-geodesic segment $[\go, p_{c'}]_{c'} \cup [p, p_{c'}]$, and a quasi-geodesic ray $[p, p_{c'}] \cup c''$. }
\end{figure}

\begin{theorem}\label{minimality theorem}
Suppose $G$ acts geometrically on a proper \CAT space $X$ and $|\partial_{\kappa} X| \geq 3$. For each $\aa \in \partial_{\kappa} X$, such that $G\cdot \aa$ is dense in $\partial_{\kappa}X$. 
\end{theorem}
\begin{proof} 

By Theorem 1.1 in \cite{Hamenstaedt2008RankoneIO}, the action of $G$ on its visual boundary $\partial X$ is minimal. In particular, no element in the visual boundary is a global fixed point of $G$. Now, since elements of the $\kappa$-boundary $\partial_\kappa X$ are also elements in the visual boundary, we conclude that $G.\aa \neq \aa$ for any $\bfa \in \partial _\kappa X.$

 Fix a point $\bfa \in \partial_\kappa X$, by the above, there exists a group element  $g \in G$ such that $g \bfa \neq  \bfa$. Sine $\pka X$ is a visibility space \cite{Zalloum_2022}, we can choose a bi-infinite geodesic line, denoted $\ell$, that connects $\bfa$ and $g \cdot \bfa$ $\pka G$, we will write $\ell^+$ for $\bfa$ and $\ell^-$ for $g \cdot \bfa$. Given any point $\bb \in \partial_\kappa X$ (not necessarily different from $\ell^+$ or $\ell^-$), and let $\ub$ be a geodesic ray representing $\bb$ with $\ub(0)=\go$. It suffices to show that a subset of the points in $G\cdot \aa$ converges to $\ub$. If $\bb = \ell^+$ or $\bb = \ell^-$ then
we are done. Otherwise fix a point $p \in \ell$.

 Since $G$ acts on $X$ cocompactly, there exist a constant $C$ and elements $\{g_i \in G, i=1,2, 3,... \}$ such that
 $d(g_i \cdot p, \ub(i)) \leq C \, \text{ for all } i \in \mathbb{N}$.

  
  \begin{figure}[!h]
\begin{center}
\begin{tikzpicture}[scale=0.7]


\draw[thick]  (0,-.3) -- ++(.1,3.3) node[above] {$\underline{b}$};

\node[below, black] at  (0,-.3) {$\go$};







\draw[thick,fill=black] (-.2,.3) circle (0.03cm);

\draw[thick,fill=black] (.2,.5) circle (0.03cm);
\draw[thick,fill=black] (-.17,.8) circle (0.03cm);

\draw[thick,fill=black] (.19,1.28) circle (0.03cm);

\draw[thick,fill=black] (-.2,1.8) circle (0.03cm);

\draw[thick,fill=black] (.17,2.1) circle (0.03cm);

\draw[thick,fill=black] (-.17,2.5) circle (0.03cm);

\draw[<->, thick,red] ({-3*cos(15)},{3*sin(15)}) .. controls ++({3*cos(15)},{3*sin(15)})  and ++({-2*cos(15)},{2*sin(15)}) ..({3*cos(15)},{3*sin(15)}) node[right] {$\alpha_i =g_i \ell^-$};

\node[<->, left, red] at ({-3*cos(15)},{3*sin(15)}) {$\beta_i = g_i \ell^+$};

\node[right, black] at (.19,1.5) {$p_i$} ;

\draw[blue]  (0,-.3) -- ++(.2,1.6);





\end{tikzpicture}
\end{center}
\caption{ Translates of $\ell$ by $\{g_i\}$ along $\underline{b}$}

\end{figure}

 Denote $l_i:=g_i \cdot \ell$. Also let $p_i:=g_i \cdot p$ and consider the quasi-geodesics \[ \alpha_i:=[\go, p_i] \cup [p_i, l^+_i), \beta_i:=[\go, p_i] \cup [p_i, l^-_i).\] Since $l_i$ is a line in a \CAT space, at least one of the local angles $\angle(\go, p_i,l_i(\infty))$, $\angle(\go, p_i,l_i(-\infty))$ is greater than or equal $\frac{\pi}{2}.$ Consequently, at least one of $\alpha_i$ or $\beta_i$ is a $(3,0)$-quasi-geodesic ray. For each $i$, let $\gamma_i \in \{\alpha_i, \beta_i\}$ be such a $(3,0)$-quasi-geodesic ray. Each $\gamma_i$ is $\kappa$-Morse as its tail is $\kappa$-Morse. By Lemma~\ref{concatenationofgauges}, for all $i$ and for all $q,Q$ we have, \[m_{\gamma_i}(q, Q)<m'_b(3q, Q)+m_l(3q, Q)+C',\]  where $C'$ is a constant depending only on $C.$ Since $m_{\gamma_i}$ does not depend on $i$ we write it as $m_\gamma$.  Let $q,Q$ be small compared to $r$. Let 
\[ \kappa'=3m_\gamma C \kappa+C, \]
Since $\ub$ is $\kappa$-Morse I, we have that for each $r>0$, and for each pair of $(q, Q)$ small compared to $r$,  there exists $R(q, Q, r, \kappa') \geq 1$ such that the conclusion of the $\kappa$-Morse I notion holds. Furthermore, by the proof of Theorem 3.14 in \cite{QR19}, since $\kappa'$ is concave, if $R=
R(q, Q, r, \kappa')$ satisfies the definition of $\kappa$-Morse I, then all $R > R(q, Q, r, \kappa')$ also satisfies the definition of $\kappa$-Morse I.

Recall that $\gamma_i \in \{\alpha_i, \beta_i\}$. Let $i = \lceil R \rceil$. By construction we have $d(\ub_R, \gamma_{\lceil R \rceil} ) \leq C.$ Thus, there exist a point $s_ {\lceil R \rceil}\in [0,\infty)$ with $d(\ub_R, \gamma_{\lceil R \rceil}(s_{\lceil R \rceil})) \leq C.$ In particular, since $\gamma_{\lceil R \rceil}$ is a $(3,0)$-quasi-geodesic, we have 

$$R-C \leq d(\gamma(s_{\lceil R \rceil}), \go) \leq R+C \, \Longrightarrow \, s_{\lceil R \rceil} <3(R+C)$$

Now it remains to show that $\gamma_{\lceil R \rceil} \in \calU(\bb, r)$. Let $\zeta$ be a $(q,Q)$-quasi-geodesic in the class of $\gamma_{\lceil R \rceil},$ then by \cite[Lemma 3.4]{QRT22} , we get that 
\begin{align*}
d(\gamma_{\lceil R \rceil}(s_{\lceil R \rceil}),\zeta) &\leq  m_\gamma \kappa(s_i) &\text{as $\zeta$ and $\gamma_{\lceil R \rceil}$ are both $\kappa$-Morse.}\\
                                     &\leq m_\gamma\kappa(3(R+C)) &\text{ as $\kappa$ is monotone nondecreasing.}\\
                                     &\leq m_\gamma\kappa(3CR) \leq 3m_\gamma C \kappa(R) &\text{as $\kappa$ is convex.}
                                     \end{align*} 
                                     
                                     This provides a point $x \in \zeta$ with $d(x, \gamma_i(s_i)) \leq m_\gamma C \kappa(R).$ Hence, by the triangle inequality, we get that 
                                     $d(x, b(R)) \leq 3m_\gamma C \kappa(R)+C.$ Now, recall that $\kappa'=3m_\gamma C \kappa+C$, thus $R(r, q, Q, \kappa')$ is precisely that
                                     \[
                                     d(x, b(R)) \leq 3m_\gamma C \kappa(R)+C \Longrightarrow \zeta_r \subseteq \calN_\kappa(\ub, r).
                                     \]
                                     
                                     This holds for every $(q, Q)$-quasi-geodesic representative of $\gamma_i$ and thus we have 
\begin{equation}
\gamma_{\lceil R \rceil} \in \calU(\bb, r)
                                     \end{equation}
                                     That is to say, for larger and larger $r$ we can find an associated sequence of $\gamma_i$ that is in $\calU(\bb, r)$, Thus up to a subsequence $\gamma_i = g_i \cdot \aa$ limits to $\bb$. 

\end{proof}

As a consequence of of the proof we can establish the following:
\begin{corollary}\label{Cork}
Let $\bb$ be a bi-infinite axis of a rank-one element $g \in G$ and let $\bfa$ be another element in $\pka X$ (not necessarily different from the ends of $\bb$). Then for each $r$, there exists a large
enough $n$ such that $g^k \cdot \bfa  \in \calU([(\bb^{+\infty}], r)$ for all $k \geq n$. 
\end{corollary}

\begin{proof}
As established in the proof of Theorem~\ref{minimality theorem}, for a sequence of larger and larger $r$ one can construct as in the proof an associated sequence of $\gamma_i$ that is in $\calU(\bb, r)$. Each $\gamma_i$ has in its class a  $(3, 0)$-quasi-geodesic ray with an initial geodesic segment that is $[\go, g_i \cdot p]$, where the sequence 
\[
\{g_1, g_2, g_3,... \} 
\]
tracks $\bb$.
Since $\bb$ is rank-one, the tracking sequence becomes
\[
\{g, g^2, g^3, g^4... \} 
\]
thus $\gamma_i$ consists of a initial segment $[\go, g^i \cdot p]$. Thus for each $r$ there exists  $\lceil R(r) \rceil$ such that if $i = \lceil R(r) \rceil$ then 
\[
[\gamma_i] \in \calU(\bb, r).
\]
Furthermore, by the proof of Theorem~\ref{minimality theorem}, all $R> R(r)$ works for the definition of $\kappa$-Morse I, thus 
for all $i \geq \lceil R(r) \rceil$, we also have that 
\[
[\gamma_i] \in \calU(\bb, r).
\]
Combine the two we then have that there exists a large
enough $n$ such that $g^k \cdot \bfa  \in \calU([(\bb^{+\infty}], r)$ for all $k \geq n$.
\end{proof}

\begin{remark}\label{rmk:North_South}
Consider a point $\bb \in \partial_\kappa X$ represented by a geodesic ray $b$. Let $p \in X$, $g_i \in G$ and $C>0$ be such that every point $b(t)$ is within $C$ of $g_i \cdot p$ for some $i.$  The above argument shows that for any pair of points $\bfa,\bfa' \in \partial_\kappa X$, the orbit $g_i\{\bfa,\bfa'\}$ has a subsequence converging to $\bb.$ Suppose further that $\bfa \neq \bb$. If we are provided that $g_i\bfa=\bfa$ for all $g_i$, then $g_i\bfa'$ has a subsequence converging to $\bb.$ This observation will be used in the proof of Theorem ~\ref{NSdym}.
\end{remark}

Now we prove a weak version of North-South dynamics for the action of a group on its $\kappa$-boundaries.
\begin{theorem}\label{NSdym}
Let $g \in G$ be a rank-one element. For an open set $U$ containing $g^{\infty}$ and a compact set $\calC \in (\pka G \setminus [g^{-\infty}])$, there exists an $N$ such that 
for all $n \geq N$, we have $g^n \cdot \calC \subset U$.
\end{theorem}
\begin{proof} Fix an open set $U$ containing $g^{+\infty}$ and let $\calU(g^{+\infty}, r)$ be in $U.$
Let $A$ be a line connecting $g^{\infty}$ and $g^{-\infty}$. That is, $A$ is an axis of $g$ and $\{g^i \}$ tracks $A$. For a point $\aa \in \calC \subset \pka X,$ we let $l$ be a geodesic line connecting $[g^{-\infty}]$ to $\aa$ and fix a point $p$ on $l$. let this point $p$ be a new basepoint. We can do this because $\pka G$ is basepoint invariant there exists a natural homeomorphism of the boundaries when changing basepoint.  Define $C:=d(p, A)$, in particular, since $g$ preserves the line $A,$ we have $d(g^np, gA)=d(g^np,A)=C.$ Hence, for each $n$ the quasi-geodesic ray \[\gamma_n:=[\go, g^n \cdot p] \cup [g^n\cdot p, g^n \cdot \aa)\] is a $(3,0)$-quasi-geodesic.

   \begin{figure}[H]
\begin{tikzpicture}[scale=0.8]
 \tikzstyle{vertex} =[circle,draw,fill=black,thick, inner sep=0pt,minimum size=.5 mm]
 
[thick, 
    scale=1,
    vertex/.style={circle,draw,fill=black,thick,
                   inner sep=0pt,minimum size= .5 mm},
                  
      trans/.style={thick,->, shorten >=6pt,shorten <=6pt,>=stealth},
   ]

  \node (o) at (0,0)  [label=left:$g^{-\infty}$] {}; 
  \node (o1) at (13, 0)[label=right:$g^{-\infty}$] {}; 
    \node(o2) at (3, 3)[label=above:$\aa$] {}; 
    
    \node[vertex] at (2.33, 1)[label=above:$p$] {};
     \node[vertex] at (11.26, 1)[label=right:$g^n \cdot p$] {};
      \node at (12, 3)
    [label=above:$g^n \cdot \aa$] {};
    \node at (11, 0)
    [label=below:$A$] {};
    \node at (2.7, 2)
    [label=right:$l$] {};
    \node at (2.3, 0.5)
    [label=right:$C$] {};
     \draw[thick, dotted] (2.33, 1)--(2.33,-0.1){};
     \draw[thick] (o)--(o1){};
     \draw[thick] (o)to [bend right=45](o2){};
     \draw[thick] (3,0)to [bend right=45](6,3){};
       \draw[thick] (6,0)to [bend right=45](9,3){};
       \draw[thick] (9,0)to [bend right=45](12,3){};

  \end{tikzpicture}
  \caption{The quasi-geodesic ray $c'$ is covered by the union of a quasi-geodesic segment $[\go, p_{c'}]_{c'} \cup [p, p_{c'}]$, and a quasi-geodesic ray $[p, p_{c'}] \cup c''$. }
\end{figure}

By Corollary~\ref{Cork}, we have for each $R$, there exists a large enough $n$ with $g^k \cdot \aa \in \calU([g^{+\infty}], R)$ for all $k \geq n$. This holds for each point in $\calC$, i.e. for each point $
\aa \in \calC,$ there exists a large enough power $n_{\aa}$ with $g^{n_{\aa}} \cdot \aa \in \calU([g^{+\infty}], R) \subseteq U.$ Now, notice that
\[g^{n_{\aa}} \cdot \aa \in \calU([g^{+\infty}], R) \Longrightarrow \aa \in (g^{n_{\aa}})^{-1} \calU([g^{+\infty}], R).
\]

Since $g$ is rank-one, $(g^{n_{\aa}})^{-1} = g^{-n_{\aa}}$. We denote the open sets as 
\[U(n, \aa):=g^{-n_{\aa}} \calU(g^\infty, R).\]
Hence, the collection $\{U (n, \aa)| \aa \in \calC\}$ forms a cover for $\calC$ yielding a finite subcover 

$$\{U(n_i, \aa_i)\}^{i=1,2,3...m}.$$ 

Now, choose $N:=\max\{n_1, \cdot \cdot \cdot n_m\}$, and let $n$ be any natural number such that $n \geq N$. We get 
\begin{align*}
g^n \calC &\subseteq g^n (\bigcup_i U(n_i, \aa_i)) &\text{since $\bigcup_i U(n_i, \aa_i)$ is a cover.}\\
          &=g^n (\bigcup_i g^{-n_{\aa_i}} \calU([g^{+\infty}], R)) &\text{Definition of $U(n_i, \aa_i)$.}\\
          & \subseteq \bigcup_i g^{m_i} \calU(g^\infty, R) &\text{Since $n \geq N \geq n_i$. $m_i = n-n_i$.}\\
          & \subseteq \calU(g^\infty, R)\\
          & \subseteq U.\\
\end{align*}
concluding the proof.
\end{proof}

\section{Compact type $\kappa$-boundaries for \CAT  spaces}

In \cite{QR19}, it is shown that if $G = \mathbb{Z}^{2} \star \mathbb{Z}$, then $\pka G$ is not compact. In this section we show that the $\kappa$-boundary is compact if and only if set of rays in this boundary lies in a subspace that is $\delta$-hyperbolic and proper. If $G$ is $\delta$-hyperbolic, then this compact boundary coincide with the Gromov boundary. For a more interesting example, see \cite{Ber} where there is a quasi-isometric copy of an embedded hyperbolic plane in a non-hyperbolic group. 

Consider the subset of all $\sD$-strongly contracting geodesic rays emanating from $\go$ in a \CAT space $X$. The following lemma states that equipping this subset with the subspace topology of the visual boundary or the subspace topology of the $\kappa$-boundary yields homeomorphic spaces. The intuitive reason for this is the following: since quasi-geodesics stay uniformly close to $\sD$-strongly contracting geodesics, the topology of fellow travelling of geodesics (the visual topology) and the topology of fellow travelling of quasi-geodesics (the topology of the $\kappa$-boundary) coincide.

Recall that we use $\partial^{\sD}_{v}X$  to denote the set of all $\sD$-contracting geodesic rays emanating from $\go$ when equipped with the subspace topology of the visual boundary , and use $\partial^{\sD}_{\kappa}X$ when equipped with the subspace topology of the $\kappa$-boundary.

 \begin{proposition} \label{key to quasi-mobius maps}
The identity map $id:\partial^{\sD}_{v}X  \to \partial^{\sD}_{\kappa}X$ is a homeomorphism.

 \end{proposition}

\begin{proof}

 We need to show that the map $id:\partial^{\sD}_{v}X  \to \partial^{\sD}_{\kappa}X$ is a homeomorphism. Since $\partial^{\sD}_{v}X$ is closed (Lemma 3.2 in \cite{ChSu2014}) and $X$ is proper, $\partial^{\sD}_{v}X$ must be compact. Also, the space $\partial^{\sD}_{\kappa}X$ is metrizable by Theorem D in \cite{QR19}. Hence, it suffices to show that the map $id$ is a continuous map. Notice that since every geodesic ray in $\partial^{\sD}_{v}X$ is $\sD$-strongly contracting for the same $\sD$, applying Definition~\ref{Def:Neighborhood}, we get an associated Morse function  such that every geodesic ray  $\beta_0$ is $\mm(\sD)$-Morse, where $\mm(\sD)$ depends only on $\sD$ and satisfies the following: For every constants $\rr>0$, $\nn>0$ and every
sublinear function $\kappa'$, there is an $\sR= \sR(\beta_0, \rr, \nn, \kappa')>0$ where the 
following holds: Let $\eta \from [0, \infty) \to X$ be a $(\qq, \sQ)$--quasi-geodesic ray 
so that $\mm_{\beta_0}(\qq, \sQ)$ is small compared to $\rr$, let $t_\rr$ be the first time 
$\Norm{\eta(t_\rr)} = \rr$ and let $t_\sR$ be the first time $\Norm{\eta(t_\sR)} = \sR$. Then
\[
d_X\big(\eta(t_\sR), \beta_0) \leq \nn \cdot \kappa'(\sR)
\quad\Longrightarrow\quad
\eta[0, t_\rr] \subset \calN_{1}\big(\beta_0, \mm_{\beta_0}(\qq, \sQ)\big)\subset \calN_{\kappa}\big(\beta_0, \mm_{\beta_0}(\qq, \sQ)\big). 
\]. 

\noindent We first claim the following:

 \begin{claim*}
 Given $\bfb \in \partial^{\sD}_{\kappa} X,$ each neighbourhood of $\bfb$, denoted $\calU_{\kappa}\big(\bfb, \rr)$, must contain a visual neighbourhood basis of $\beta_{0}$, the unique geodesic ray in the class of $\bfb$. 
 \end{claim*}
 \begin{proof}
  To see this, let $\beta_{0} \in \bfb$ be the unique geodesic ray starting at $\go$. We wish to show that for any $\rr>0$, there exists $\rr'$ and $\epsilon$ such that $\calU_{v}\big(\beta_{0}, \rr',\epsilon) \subseteq \calU_{\kappa}\big(\bfb, \rr)$

  In other words, we want to show that for any $\rr>0$, there exists $\rr'$ and $\epsilon$ if a geodesic ray $\alpha_{0} \in \bfa$ with $\alpha_{0}(0)=\go$ satisfies $d(\alpha_{0}(t), \beta_{0}(t))<\epsilon$ for $t \leq \rr'$, then, any $(\qq,\sQ)$-quasi-geodesic representative $\alpha$ of $\aa$ with $\mm_{\beta_0}(\qq,\sQ)$ small compared to $\rr$, we have 
\[\alpha|_{\rr}  \subset \calN_\kappa(\zeta, \mm_{\beta_0}(\qq,\sQ)).\]Remember that  $\alpha|_{\rr} = \alpha([0,t_{\rr}])$ where $t_{\rr}$ is the first time where $||\alpha(t)|| = \rr$.

Let $\rr$ be given and let \[n=\max \{\mm_{\beta_0}(\qq,\sQ)+1 | \qq,\sQ \leq \rr\}.\] 
By Definition~\ref{Def:Neighborhood}, with $Z=\beta_0$, there exists an $R=R(\rr,n)$ such that any $(\qq,\sQ)$-quasi-geodesic representative $\beta$ of $\aa$ with $\mm_{\beta}(\qq,\sQ)$ small compared  to $\rr$, we have 
\[d(\beta(t_R),b)<n \Longrightarrow \beta_\rr  \subset \mathcal{N}_{1}(\beta_0, \mm_{\beta}(\qq,\sQ)).\]
 Choose $\rr'=\rr+R$ and $\epsilon=1$. Hence, we want to show that if $d(a(t), b(t))<1$ for $t \leq \rr+R$, then $\beta_\rr  \subset \mathcal{N}_\kappa(b, \mm_{\beta_0}(\qq,\sQ))$ for $\beta$ defined above. Since $a$ is 1-Morse with gauge $\mm,$ the Hausdorff distance between $a$ and $\beta$ is at most $\mm(\qq,\sQ).$ This implies that for any $0<t\leq \rr+R$, we have 
 \[d(a(t)), \beta (i_t))<\mm_{\beta_0}(\qq,\sQ),\]
  for some $i_t.$ 
  
  Therefore, if $t_R$ is the first time with $\Norm{\beta(t_R)}=R$, we must have \[d(a, \beta(t_R)) <\mm_{\beta_0}(\qq,\sQ).\]

   Now, since $d(a(t),b(t))<1$ for all $t < \rr+R$ and as $d(a, \beta(t_R)) <\mm_{\beta_0}(\qq,\sQ),$ the triangle inequality gives 
   \[d(b,\beta(t_R)) \leq d(b,a)+d(a,\beta(t_R))\leq 1+\mm_{\beta_0}(\qq,\sQ),\] which we can rewrite as \[\beta_r  \subset \mathcal{N}_1(b, \mm_{\beta_0}(\qq,\sQ))\subset \mathcal{N}_\kappa(b, \mm_{\beta_0}(\qq,\sQ))\] which proves the claim. \end{proof} 
   Now we are left to show that the map $id$ is continuous. Let $\{\cc_n\}, \cc \in \partial^{\sD}_{v}X$ with $\cc_n \rightarrow \cc$. Assume that $\cc_n \rightarrow \cc$ in $\partial^{\sD}_{v}X$, we want to show that $\cc_n \rightarrow \cc$ in  $\partial^{\sD}_{\kappa}X$. Using the above claim, since each neighbourhood of $\cc$ in  $\partial^{\sD}_{\kappa}X$ contains an open neighborhood of $\partial^{\sD}_{v}X$, the statement is immediate. \end{proof}

\begin{corollary}\label{Morse boundary embedds} For a \CAT space $X$, the natural map $i:\partial_{\star}X\hookrightarrow \partial_{\kappa}X$ is continuous.

\end{corollary}

\begin{proof}
By Lemma \ref{key to quasi-mobius maps},  $i_{\sD}:\partial^{\sD}_{v} X\hookrightarrow \partial_{\kappa}X$ is continuous for each $\sD$. Since \[\partial_{\star} X=\varinjlim \partial^{\sD}_{v} X,\] by definition $i:\partial_{\star}X\hookrightarrow \partial_{\kappa}X$ is continuous.
\end{proof}

%
%

\begin{corollary} \label{hyperbolic implies compactness}
Let $X$ be proper \CAT hyperbolic space and let $\kappa$ be a sublinear function. The space $\partial_{\kappa} X$ is compact, and the  $\kappa$-Morse boundary is homeomorphic to 
the Gromov boundary.

\end{corollary}

\begin{proof}
Since $X$ is a hyperbolic space,  there exists a uniform constant $\sD$  such that every geodesic ray is $\sD$-strongly contracting. This implies that the subspace $\partial^{\sD}_{v}X$ defined above is the entire visual boundary, in other words, we have   $\partial^{\sD}_{v}X = \partial_{v} X.$ Also, since every geodesic ray is $\sD$-strongly contracting, the subspace $\partial^{\sD}_{\kappa}X$ defined above is the full $\kappa$-boundary as a topological space. That is to say, $\partial^{\sD}_{\kappa}X= \partial_{\kappa} X$ as topological spaces. Proposition \ref{key to quasi-mobius maps} then yields a homeomorphism between the visual boundary of $X$, $\partial_{v} X$ and the $\kappa$-boundary of $X$,  $\partial_{\kappa} X.$  Therefore as topological spaces, we have
\[
\partial X \cong \pka X.
\]

Since $X$ is proper, $\partial X$ is compact, thus the $\kappa$-boundary $\partial_{\kappa} X$ must also be compact.
\end{proof}

\begin{theorem} \label{main theorem of the section}
Suppose a group $G$ acts geometrically on a \CAT space $X$ such that $\partial_{\kappa}X \neq \emptyset$, then the following are equivalent: 

\begin{enumerate}
    \item Every geodesic ray in $X$ is $\kappa$--contracting.
    \item Every geodesic ray in $X$ is strongly contracting.
    \item $\partial_{\kappa} X$ is compact.
    \item The space $X$ is hyperbolic.
\end{enumerate}

\end{theorem}
\begin{proof}
We start by showing (3) implies (1). The statement is vacuously true if $\pka X$ is empty. If  $\pka X$ is non-empty, then by \cite{Zalloum_2022}, there exists a rank one isometry $g$. This yields the existence of a strongly contracting geodesic line $l_g$ that is an axis for $g$. Let $\go$ be a point on $l_g$ and let $\beta$ be an arbitrary geodesic ray emanating from $\go$. We show now that $\beta$ is  $\kappa$-contracting. Since the action of $G$ on $X$ is cocompact, there exists a $C \geq 0$ and a sequence of group elements $\{g_i\} \subseteq G$ such that $d(\beta(i), g_i \cdot \go) \leq C$ for each $i \in \mathbb{N}$ (the black dots in Figure \ref{Translates of beta}). Now, consider the sets given by $g_i l_g$. Since $g_{i}$ acts by isometry, these are bi-infinite geodesic lines passing the points $g_i\go$. Recall $[\cdot, \cdot]$ denote a geodesic segment between two points. By \CAT geometry, the concatenation of two geodesic segments at angle bounded below by $\pi/2$ forms a (3,0)-quasi-geodesic segment. Lastly, we denote one end of 
$g_{i} l_g$ by $g_{i} l_g(\infty)$ and the other end by  $g_{i} l_g(-\infty)$.

For each $i$, consider the concatenation $$[\go,g_i\go] \cup [g_i\go, g_{i} l_g(\infty)] \text{ and } [\go,g_i\go] \cup [g_i\go, l_g(-\infty)].$$ By \CAT geometry, one of these two concatenations consists of 
geodesic segments intersecting at angles bounded below by $\pi/2$.  Thus one of the two concatenations is a (3,0)-quasi-geodesic ray starting at $\go$.  Relabel the sequence of (3,0)--quasi-geodesic rays defined by concatenating $[\go,g_i\go]$ with either $[g_i\go, l_g(\infty)]$ or $[g_i\go, l_g(-\infty)]$ to form a sequence of (3,0)--quasi-geodesic rays,  by $\gamma_{i}$.  Since $\partial_{\kappa} X$ is compact, up to passing to a subsequence, $[\gamma_i]$ converges to an element $\bb \in \partial_{\kappa}X$. Let $b$ be the geodesic representative of $\bb$. The convergence implies that  for each $\rr>0,$ there exists $k$ such that if $i \geq k,$ the sequence $\gamma_i$ satisfies

\[
\gamma_i|_{\rr} \subset \calN_{\kappa}\big(b, \mm_{b}(3,0)\big).\]

 Since the subsegment of $\gamma_i$ given by $[\go, g_i\go]$ is in a $C$-neighbourhood of $\beta$, we have for each $r$, $\beta|_{r}$ is in $\calN_{\kappa}\big(b,C+ \mm_{b}(3,0)\big),$ and hence \[ \beta \in \calN_{\kappa}\big(b,C+ \mm_{b}(3,0)\big).\] Lemma \ref{uniquegeodesic} then implies that $\beta=b$. Thus $[\beta] = [b ]= \bb \in \pka X$, which finishes the proof.

\begin{figure}[!h]
\begin{center}
\begin{tikzpicture}[scale=0.7]

\draw[very thick,black] (0,0) circle (3cm);

\draw[thick]  (0,-.3) -- ++(.1,3.3) node[above] {$\beta$};

\node[below, black] at  (0,-.3) {$\go$};




\draw[<->, thick,red] ({-3*cos(-15)},{3*sin(-15)}) .. controls ++({3*cos(-15)},{-3*sin(-15)})  and ++({-2*cos(-15)},{-2*sin(-15)}) ..({3*cos(-15)},{3*sin(-15)}) node[right] {$l_g(\infty)$};

\node[left, red] at ({-3*cos(-15)},{3*sin(-15)}) {$l_g(- \infty)$};

\draw[thick,fill=black] (-.2,.3) circle (0.03cm);

\draw[thick,fill=black] (.2,.5) circle (0.03cm);
\draw[thick,fill=black] (-.17,.8) circle (0.03cm);

\draw[thick,fill=black] (.19,1.28) circle (0.03cm);

\draw[thick,fill=black] (-.2,1.8) circle (0.03cm);

\draw[thick,fill=black] (.17,2.1) circle (0.03cm);

\draw[thick,fill=black] (-.17,2.5) circle (0.03cm);

\draw[<->, thick,red] ({-3*cos(15)},{3*sin(15)}) .. controls ++({3*cos(15)},{3*sin(15)})  and ++({-2*cos(15)},{2*sin(15)}) ..({3*cos(15)},{3*sin(15)}) node[right] {$g_il_g(\infty)$};

\node[<->, left, red] at ({-3*cos(15)},{3*sin(15)}) {$g_i l_g(- \infty)$};

\node[right, black] at (.19,1.5) {$g_i \go$} ;

\draw[blue]  (0,-.3) -- ++(.2,1.6);





\end{tikzpicture}
\end{center}
\caption{ Translates of $l_{g}$ by $\{g_i\}$ along $\alpha_{0}$}\label{Translates of beta}

\end{figure}

Next we show that (1) implies (4). If every geodesic ray is $\kappa$-contracting, then $X$ does not contain an isometric copy of $\mathbb{E}^2,$ and hence, by the Flat Plane Theorem (\cite{BH1} III.$\Gamma$.3 Theorem 3.1), the space $X$ must be hyperbolic. The implication $(4) \Rightarrow (3)$ is Corollary \ref{hyperbolic implies compactness}. 

Lastly, we prove the equivalence between (2) and (4). Since every geodesic ray is $\sN$-Morse for the same $\sN$ in a $\delta$-hyperbolic space, we have $(4) \Rightarrow (2)$.
On the other hand, by way of contradiction, suppose $X$ is not a hyperbolic space, then it must contain isometrically a copy of $\mathbb{E}^2$ by the Flat Plane Theorem (Flat Plane Theorem (\cite{BH1} Theorem 3.1). Let $\go \in \mathbb{E}^2$ and the geodesic rays that stays entirely in the  is not $\sD$--strongly contracting for any $\sD$. Therefore, $(2) \Rightarrow (4)$.
\end{proof}

\begin{remark}
As we can see in \cite{Behrexample}, one can have part of the sublinearly Morse boundary being compact while the rest of the sublinearly Morse boundary is not compact. In this case the theorem can be applied to a part of the space $X$ whose sublinear boundary is the compact portion.
\end{remark}

 \section{successively quasi-m{\"o}bius homeomophisms on the $\kappa$-boundaries}\label{morsedefinitions}

In \cite{Paulin1996}, the author characterizes homeomorphisms between boundaries of cocompact hyperbolic spaces that are induced by quasi-isometries. They characterize such homeomorphisms as the ones that are \emph{quasi-m{\"o}bius}. In this section, as an application of visibility \cite{Zalloum_2022} and using work of \cite{RuthDevin},\cite{Charney2019}, we prove a weaker version of this characterization:

\begin{theorem} \label{quasimobius}
Let $X,Y$ be proper cocompact \CAT spaces with at least 3 points in their sublinear boundaries. A homeomorphism $f:\partial_\kappa X \to \partial_\kappa Y$ is induced by a quasi-isometry $h:X \to Y$ if and only if $f$ is stable and successively quasi-m{\"o}bius.
\end{theorem}

The following is an immediate consequence.

\begin{corollary}
Let $G$ and $H$ be \CAT groups. Then $G$ is quasi-isometric to $H$
if and only if there exists a homeomorphism $f : \partial_\kappa G \rightarrow \partial_\kappa H$ which is successively quasi-m{\"o}bius and stable.
\end{corollary}

\subsection{Contracting geodesic rays and quasi-m{\"o}bius maps}\label{morsedefinitions} We remind the reader of a few terminologies from \cite{RuthDevin} and \cite{Charney2019}. Recall that a geodesic $\gamma$ is \emph{strongly contracting} if it is in the sublinearly Morse boundary whose associated sublinear function $\kappa = 1$. This implies the existence of a constant $\sD$ such that all disjoint balls project onto $\gamma$ to a set of diameter at most $\sD$, in which case we say $\gamma$ is \emph{$\sD$-strongly contracting}. Consider the set of all $\sD$-strongly-contracting geodesic rays emanating from $\go.$ We can think of this set as a subspace of the various boundaries we study in this paper: we use $\partial^{\sD}_{v}X$  to denote the set of all $\sD$-contracting geodesic rays emanating from $\go$ when equipped with the subspace topology of the visual boundary , and use $\partial^{\sD}_{\kappa}X$ when equipped with the subspace topology of the $\kappa$-boundary.


 Denote by $\partial^{(n,\sD)}_{\kappa}X$ the collection of all $n$-tuples $(a_1,a_2,...,a_n)$ of distinct points $a_i \in \partial_{\kappa} X $ such that every bi-infinite geodesic connecting $a_i$ to $a_j$ is $\sD$-strongly contracting.

\begin{definition}
Let $X,Y$ be proper geodesic \CAT space.

\begin{itemize}
    \item  A map $f: \partial_{\kappa} X
\rightarrow \partial_{\kappa} Y$ is said to be \emph{$1$-stable} if for every $\sD$, there exists $\sD'$ such that $f(\partial^{\sD}_{\kappa}X)\subseteq \partial^{\sD'}_{\kappa}Y.$

\item A map $f: \partial_{\kappa} X \to \partial_{\kappa} Y$ is said to be \emph{$2$-stable} if for every $\sD$, there exists $\sD'$ such that \[f(\partial^{(2,\sD)}_{\kappa}X)\subseteq \partial^{(2,\sD')}_{\kappa}Y.\] 
\end{itemize}
\end{definition}

Notice that it follows from the above definition that a 2-stable map $f$ maps $\partial^{(n,\sD)}_{\kappa}X$ to $\partial^{(n,\sD')}_{\kappa}X$ for all $n \geq 2$. Hence, it makes sense to make the following definition.

A map $f: \partial_{\kappa} X \to \partial_{\kappa} Y$ is said to be \emph{stable} if it is both 1 and 2 stable.

\begin{definition}
 The cross-ratio of a four-tuple $(a, b, c, d) \in     \partial^{(4,\sD)}_{\kappa} X$ is defined to be $[a,b,c,d]=$ $\ou$ $\underset{\alpha \in (a,c)}{\sup}d(\pi_{\alpha}(b), \pi_{\alpha}(d))$, where the sign is positive if the orientation of the geodesic $(\pi_{\alpha}(b), \pi_{\alpha}(d))$ agrees with that
of $(a, c)$ and is negative otherwise.

A stable map $f:\partial_{\kappa}X \rightarrow \partial _{\kappa}Y$ is said to be \emph{successively quasi-m{\"o}bius} if for every $\sD,$ there exists a continuous map $\psi_{\sD}:[0, \infty) \rightarrow [0, \infty)$ such that for all 4-tuples $(a, b, c, d) \in     \partial^{(4,\sD)}_{\kappa} X$, we have $[f(a), f(b), f(c), f(d)] \leq \psi_{\sD}(|[a,b,c,d]|)$.

Let  $X_1 \subset X_2 \subset X_3 \subset...$
 be a nested sequence of topological spaces. The \emph{direct limit} of $\{X_i\}$, denoted by $\varinjlim X_i$, is the space consisting of the union of all $X_i$ given the following topology: A subset $U$ is open in $\varinjlim X_i$ if $U \cap X_i$ is open in $X_i$ for each $i.$ 
 
 \end{definition}
The following is a standard way to establish continuous maps between two nested sequences and the proof is left as an exercise for interested readers.


\begin{lemma} \label{lem: continuous maps between direct limits}
Let $\{X_i\}$, and $\{Y_j\}$ be two sequences of nested topological spaces. Let 
\[X =  \varinjlim X_i \text{ and } Y = \varinjlim Y_i \] 
be the direct limit of $\{ X_i \}$ and $ \{ Y_i \}$ respectively. 
If $f:X\to Y$ is a map such that: \begin{itemize}
    \item For each $i$ there exists some $j$ with $f(X_i) \subseteq Y_j$.
    \item $f|_{X_i}:X_i \to Y_j$ is continuous.
\end{itemize}

Then $f$ is continuous.
\end{lemma}

Consider the topological spaces $\partial^{\sD}_{v}X$. The \emph{Morse boundary} $\partial_{\star} X$ is the direct limit of the topological spaces $\partial^{\sD}_{v} X$ where $\sD \in \mathbb{N}.$ In other words  $$\partial_{\star} X=\varinjlim \partial^{\sD}_{v} X$$ Hence, a set $U$ is open in  $\partial_{\star} X$ if and only if $U \cap \partial^{\sD}_{v}X$ is open for each $\sD$. We will often make  use of the following theorem.

\begin{theorem}[\cite{Charney2019}] \label{ruthquasimobius}
Let $X,Y$ be proper cocompact \CAT spaces with at least 3 points in their Morse boundaries. A homeomorphism $f:\partial_{\star} X \to \partial_{\star} Y$ is induced by a quasi-isometry $h:X \to Y$ if and only if $f$ is is 2-stable and successively quasi-m{\"o}bius.
\end{theorem}

%
%

\begin{lemma} \label{stable} 
A quasi-isometry $h:X \to Y$ induces a stable homeomorphism $\partial_{\kappa} h: \partial_\kappa X \to \partial_\kappa Y.$
\end{lemma}

\begin{proof}
Fix $\go \in X$ and let $\go'=h(\go)$ where $h$ is a $(\kk,\sK)$-quasi-isometry. Qing and Rafi show that a quasi-isometry $h$ induces a homeomorphism $\partial h$ on their respective $\kappa$-boundaries. If $\gamma$ is a $\sD$-strongly contracting geodesic ray, then by \cite{ChSu2014}, the unique geodesic ray starting at $h(\go)$ and representing $[f(\gamma)]$ must be $\sD'$-strongly contracting where $\sD'$  depends on $\sD, \kk$ and $\sK$. This implies that $\partial_{\kappa} h$ is 1-stable. Now, Theorem \ref{ruthquasimobius} gives us that the map induced by $h$ on the Morse boundary is 2-stable. Hence, we deduce that $\partial_{\kappa} h$ is stable.

%

\end{proof}

\begin{lemma}\label{stable sublinear homeomorphisms imply contracting ones}
Any homeomorphism $f: \partial_{\kappa} X \to \partial_{\kappa} Y$ such that $f, f^{-1}$ are 1-stable induces a homeomorphism $g:\partial_{\star} X \to \partial_{\star} Y$ on their Morse boundaries, with $g(x)=f(x)$ for all $x \in \partial_{\star}X.$
\end{lemma}

\begin{proof}
Let $f: \partial_{\kappa} X \to \partial_{\kappa} Y$ be a homeomorphism such that $f$ and $f^{-1}$ are 1-stable. Notice that by Theorem E in \cite{QR19}, we have that if $\kappa'<\kappa$, then the inclusion map \[i: \partial_{\kappa'}X \to \partial_{\kappa}X\] is continuous. Taking $\kappa'=1$, yields that $i:\partial_{1} X\to \partial_{\kappa}X$ is continuous. Hence, since both $f$ and $f^{-1}$ are 1-stable, the restriction of $f$ to $\partial_1 X$ induces a homeomorphism $\overline{f}:\partial_1 X \to \partial_1 Y$, with $\overline{f}=f|_{\partial_1 X }$ where $\partial_1 X$ and $\partial_1 Y$ are given the subspace topology of the $\kappa$-boundary. Meanwhile,
 \[ \partial_1 X=
\bigcup_{\sD=1}^{\infty} \partial^{\sD}_1 X  \text{ and } \partial_1 Y=
\bigcup_{\sD=1}^{\infty} \partial^{\sD}_1 Y .\] Since $\partial^{\sD}_1 X$ is equipped with the subspace topology of $\partial_1 X$, the inclusion map \[i^{\sD}: \partial^{\sD}_1 X \hookrightarrow \partial_1 X\] is continuous. Using Lemma \ref{key to quasi-mobius maps}, we get that \[i^{\sD}: \partial^{\sD}_v X\hookrightarrow{} \partial_1 X \] is continuous for every $\sD$, where $\partial^{\sD}_v X$ is given the subspace topology of the visual boundary. Furthermore, since $f$ is 1-stable, we have $\overline{f} \circ i^{\sD}:\partial^{\sD}_v X\hookrightarrow \partial^{\sD'}_1 Y$ for some $\sD'$ where $\overline{f} \circ i^{\sD}$ is continuous. Using Lemma \ref{key to quasi-mobius maps}, we obtain a continuous map \[\overline{f} \circ i^{\sD}:\partial^{\sD}_v X\hookrightarrow \partial^{\sD'}_v Y\] for each $\sD$. Hence, by Lemma \ref{lem: continuous maps between direct limits}, we get a continuous map $g:\partial_{\star} X \to \partial_{\star}Y$. Applying the same argument above to $f^{-1}$ yields a continuous map $g':\partial_{\star} Y \to \partial_{\star}X$ with 
\[g \circ g'= id_{\partial_{\star} X} \text{ and }g' \circ g=id_{\partial_{\star} Y}, \] which finishes the proof.

\end{proof}

%
%
%

\subsection{Proof of Theorem~\ref{quasimobius}}\label{proof of quasi}

\begin{proof}
($\Rightarrow$) If $h$ is a quasi-isometry, then $f: = \partial h$ is stable by Lemma \ref{stable} . Also, $f$ is successively quasi-m{\"o}bius by Theorem~\ref{ruthquasimobius}.

($\Leftarrow$) Using Lemma \ref{stable sublinear homeomorphisms imply contracting ones} any stable homeomorphism  $f:\partial_\kappa X \to \partial_\kappa Y$ induces a homeomorphism $g:\partial_{\star} X \to \partial_{\star} Y$ on their Morse boundaries, with $g(x)=f(x)$ for all $x \in \partial_{\star}X$. Since $f$ is successively quasi-m{\"o}bius and $g(x)=f(x)$ for $x \in \partial_{\star}X,$ Theorem \ref{ruthquasimobius} implies the existence of a quasi-isometry $h:X \to Y$ such that $\partial h=g:\partial_{\star}X \to \partial_{\star} Y$. We wish to show that the induced map \[\partial _{\kappa}h:\partial_{\kappa} X \to \partial_{\kappa} Y \] agrees with $f.$ Notice that as a set $\partial_{\star}X=\partial_1 X,$ where $\partial_1 X$ is the subset of $\partial_\kappa X$ consisting of equivalence classes having a strongly contracting representative. Hence, we have $\partial_{\kappa}h(x)=\partial h(x)$ for all $x \in \partial_1 X \subseteq \partial_{\kappa} X$. Now, since $\partial h=g$, and $g(x)=f(x)$ on $\partial_1 X$, we get that \[\partial _{\kappa}h(x)=\partial h(x)=g(x)=f(x) \] for all $x \in \partial_1 X.$ Therefore, $\partial_{\kappa}h(x)=f(x)$ for all $x \in \partial_1 X \subseteq \partial_{\kappa} X$. It remains to show that $\partial _{\kappa}h(x')=f(x')$ for all $x' \in \partial_{\kappa} X.$ Let $x' \in \partial_{\kappa} X$, by Corollary \ref{densesubset}, there exists a sequence $x_n \in \partial_1 X$ that converges to $x'$
\[ x_n \to x' \] in $\partial_{\kappa} X.$ Since $f$ is continuous on $\partial X_{\kappa}$, we have convergence \[
  f(x_n)=\partial _{\kappa}h(x_n) \to  f(x').\] Also, since $\partial _{\kappa}h$ is continuous on $\partial_{\kappa} X,$ we get that 
\[ \partial _{\kappa}h (x_n) \to \partial _{\kappa}h (x'). \] As $\partial_{\kappa} Y $ is Hausdorff, we obtain $\partial _{\kappa}h(x')=f(x').$
\end{proof}

This result is far from satisfying, since successively quasi-m{\"o}bius requires one to check the quasi-m{\"o}bius condition for every $\sD$, it is a much stronger condition than quasi-m{\"o}bius. Currently there is no results directly characterizing quasi-m{\"o}bius maps on the $\kappa$-boundaries.

\bibliography{bibliography1}{}
\bibliographystyle{alpha}

\end{document}